\documentclass[11pt]{article}
\usepackage{amsthm,fullpage}
\usepackage{amsmath,amssymb}
\newtheorem{thm}{Theorem}[section]
\newtheorem{lem}[thm]{Lemma}
\newtheorem{prop}[thm]{Proposition}
\newtheorem{cor}[thm]{Corollary}

\newtheorem{defn}[thm]{Definition}
\newtheorem{rem}[thm]{Remark}
\newtheorem{exmp}[thm]{Example}

\newcommand{\edge}{\mbox{\raise.2em\hbox to1.5em{\vbox{\hsize1.5em\hbox to1.5em{}\hrule}}}}
\newcommand{\noedge}{\edge\hspace*{-1.5em}\hbox to1.5em{\hss\mbox{/}\hss}}
\newcommand{\cedge}[1]{\edge\hspace*{-1.5em}\raise.4em\hbox to 1.5em{\hss{$#1$}\hss}}
\newcommand{\nocedge}[1]{\cedge{{#1}}\hspace*{-1.5em}\hbox to1.5em{\hss\mbox{\small /}\hss}}

\begin{document}
\title{A Characterization of Signed Graphs with Generalized Perfect Elimination Orderings}
\author{Koji Nuida}
\date{Research Center for Information Security (RCIS), National Institute of Advanced Industrial Science and Technology (AIST), Akihabara-Daibiru Room 1003, 1-18-13 Sotokanda, Chiyoda-ku, Tokyo 101-0021, Japan\\
k.nuida@aist.go.jp}
\maketitle
\begin{abstract}
An important property of chordal graphs is that these graphs are characterized by existence of perfect elimination orderings on their vertex sets.
In this paper, we generalize the notion of perfect elimination orderings to signed graphs, and give a characterization for graphs admitting such orderings, together with characterizations restricted to some subclasses and further properties of those graphs.

\end{abstract}

\section{Introduction}
\label{sec:intro}

An undirected graph is called {\em chordal} if any cycle with at least four vertices has a {\em chord} (an edge not in the cycle with both endpoints in the cycle).
Chordal graphs are a classical subject in graph theory and these graphs have been playing significant roles also in several related research areas.
A property used in such research frequently is that a graph is chordal if and only if it admits a special ordering of vertices, called a {\em perfect elimination ordering} or a {\em vertex elimination ordering} (see \cite[Section 7]{FG}).
Roughly speaking, perfect elimination orderings correspond to a kind of growing processes from an empty graph to the given graph, in which a new vertex is pasted to the present graph at a clique.
This characterization of chordal graphs is very significant, since it connects combinatorial properties of the graph to geometric ones.
For example, a famous result regarding hyperplane arrangements, given by Richard P.\ Stanley \cite{Sta1}, states that an arrangement parameterized by a graph in certain manner is \lq\lq free'' if and only if the corresponding graph is chordal.

The aim of this paper is to generalize the notion of perfect elimination orderings (and even the notion of chordal graphs) to signed graphs, i.e.\ graphs with each edge having a sign \lq\lq$+$'' or \lq\lq$-$'', and to give a complete characterization of a signed graph admitting such an ordering.
In this paper we call such an ordering and such a graph a {\em signed elimination ordering} and a {\em signed-eliminable} graph, respectively.
A signed elimination ordering is such that it is a usual perfect elimination ordering when restricted to edges with a fixed sign, and it satisfies a further condition across the two signs (see Definition \ref{defn:signed-elimination} for precise definition).
Then our characterization (Theorem \ref{thm:characterization}) says that a signed graph is signed-eliminable if and only if the subgraph restricted to each sign is chordal and it satisfies certain further conditions involving edges with both signs.
The characterization implies that it is indeed a generalization of the aforementioned classical equivalence of chordality to admitting perfect elimination orderings.

We give some comments on related works.
First, the present work is motivated by recent research by Takuro Abe, Yasuhide Numata and the author to generalize Stanley's aforementioned result and to give a partial solution for a conjecture proposed by Christos A.\ Athanasiadis \cite{Ath1} (more precisely, to prove the \lq\lq if'' part of Athanasiadis's conjecture).
See \cite{ANN} for details.
Secondly, a recent work by Terry A.\ Mckee \cite{McK} also extended the notion of chordal graphs to signed graphs.
However, his generalization was done in a very different manner from ours, and there is no obvious relation between his and ours.

This paper is organized as follows.
In Section \ref{sec:preliminary}, we present and fix notations and terminology for graphs and for signed graphs, and also give some lemmas for later references.
In Section \ref{sec:SEO}, we introduce the notion of signed elimination orderings and signed-eliminable graphs, state and prove some fundamental properties, and also give a greedy algorithm for deciding whether a given graph is signed-eliminable and constructing a signed elimination ordering (if it exists).
Section \ref{sec:characterization_preliminary} is an introduction to the full characterization of signed-eliminable graphs; we give definitions of two kinds of exceptional subgraphs (called {\em mountains} and {\em hills}), prove that any signed graph with three vertices is signed-eliminable, and present some further properties.
Section \ref{sec:characterization} is devoted to the statement and the proof of our full characterization.
Finally, in Section \ref{sec:special_case}, we give characterizations of the signed-eliminable graphs in some subclasses (graphs with four vertices; chordal graphs; graphs with independence number less than three; and complete graphs) by restricting our full characterization to these subclasses.

\paragraph*{Acknowledgments.} 
This work was originally motivated by interesting research of Dr.\ Takuro Abe and Dr.\ Yasuhide Numata, thus the author would like to express his best gratitude to them.
The contents of Section \ref{subsec:invariant} are also inspired by Abe and Numata.
Moreover, the author would like to thank every person who gave comments on this work.

\section{Preliminaries}
\label{sec:preliminary}

\subsection{Graphs}
\label{subsec:preliminary_graph}

In this paper every graph $G = (V,E)$ is finite, simple and undirected.
See any textbook of graph theory, e.g.\ \cite{graph}, for basic notations and terminology.
We denote $v \edge w$ and $v \noedge w$, respectively, to signify that $vw \in E$ and $vw \not\in E$, where $vw$ denotes the unordered pair of $v$ and $w$.
For $V' \subset V$, let $G|_{V'}$ denote the induced subgraph of $G$ with vertex set $V'$, and write $G \setminus V' = G|_{V \setminus V'}$.
In this paper, we often abbreviate a singleton $\{x\}$ simply to $x$ unless some ambiguity arises.
For $v \in V$, we write
\begin{displaymath}
N_G(v) = \{w \in V \mid v \edge w\} \mbox{ and } N_G\left[v\right] = N_G(v) \cup v \enspace,
\end{displaymath}
and for $V_1,V_2 \subset V$, define
\begin{displaymath}
\left[V_1,V_2\right] = \{vw \mid v \in V_1,\ w \in V_2,\ v \neq w\}
\end{displaymath}
(note that we do {\em not} assume that $\left[V_1,V_2\right] \subset E$).
We write $N(v) = N_G(v)$ and $N\left[v\right] = N_G\left[v\right]$ if the graph $G$ is obvious from the context.
A graph $G$ is called {\em chordal} if it has no induced cycle of length at least four.
We refer to a bijection $\nu$ from $V$ to $\{1,2,\dots,|V|\}$ as an {\em ordering} on $G$.
The following theorem is a well-known characterization of chordal graphs:
\begin{thm}
[See e.g.\ \cite{FG}]
\label{thm:chordal}
A graph $G$ is chordal if and only if there is an ordering $\nu$ on $G$ such that for any three vertices $u$, $v$ and $w$ of $G$ with $\nu(u) < \nu(w) > \nu(v)$, if $u \edge w \edge v$ then $u \edge v$.
\end{thm}
An ordering $\nu$ satisfying the condition in this theorem is called a {\em perfect elimination ordering}, or simply an {\em elimination ordering}.
Now a straightforward argument shows the following properties:
\begin{lem}
\label{lem:ElimOrd_induction}
Let $G$ be a graph and $v \in V$.
\begin{enumerate}
\item If $\nu$ is an elimination ordering on $G$ with $\nu(v) = |V|$, then $N_G\left[v\right]$ is a clique of $G$ and the restriction of $\nu$ on $V \setminus v$ is also an elimination ordering on $G \setminus v$.
\item Conversely, suppose that $N_G\left[v\right]$ is a clique of $G$ and $\nu$ is an elimination ordering on $G \setminus v$.
Then any ordering $\overline{\nu}$ on $G$ which extends $\nu$ and satisfies $\overline{\nu}(v) = |V|$ is also an elimination ordering on $G$.
\end{enumerate}
\end{lem}
We also prepare the following lemma on chordal graphs:
\begin{lem}
\label{lem:chordal_link_clique}
Let $G$ be a chordal graph and $V' \subsetneq V$ a clique of $G$.
Then there is a vertex $v \in V \setminus V'$ such that $N_G\left[v\right]$ is a clique of $G$.
\end{lem}
\begin{proof}
First, an elimination ordering $\mu$ on $G$ exists by Theorem \ref{thm:chordal}.
Let $w = \mu^{-1}(|V|) \in V$.
Then $N_G\left[w\right]$ is a clique by Lemma \ref{lem:ElimOrd_induction}(1).
Our claim holds if $w \not\in V'$; thus suppose that $w \in V'$.
If $N_G\left[w\right] = V$ (i.e.\ $G$ is a complete graph), then any vertex in $V \setminus V'$ satisfies the claim.
On the other hand, suppose that $N_G\left[w\right] \neq V$.
Then we have $N_G(w) \subsetneq V \setminus w$ and $N_G(w)$ is a clique of $G \setminus w$, therefore induction on $|V|$ enables us to take a vertex $v \in (V \setminus w) \setminus N_G(w) = V \setminus N_G\left[w\right]$ such that $N_{G \setminus w}\left[v\right]$ is a clique in $G \setminus w$.
Moreover, we have $v \noedge w$ by the choice of $v$, therefore $N_G\left[v\right] = N_{G \setminus w}\left[v\right]$ is also a clique in $G$.
Hence the claim holds, since $V' \subset N_G\left[w\right]$.
\end{proof}

\subsection{Signed Graphs}
\label{subsec:preliminary_bicolored}

A {\em signed graph} is a graph $G = (V,E)$ with a partition $E = E_+ \cup E_-$ of edge set (where each part may be empty).
For $\sigma \in \{+,-\}$, we write $G_{\sigma} = (V,E_{\sigma})$, and denote $v \cedge{\sigma} w$ and $v \nocedge{\sigma} w$, respectively, to signify that $vw \in E_{\sigma}$ and $vw \not\in E_{\sigma}$.
We simply write $N_{G_{\sigma}}\left[v\right] = N_{\sigma}\left[v\right]$ and $N_{G_{\sigma}}(v) = N_{\sigma}(v)$ if the underlying graph $G$ is obvious from the context.
In this paper, single and duplicate edges in a figure of a graph represent edges with different signs.

The following simple lemma will be used in our argument later:
\begin{lem}
\label{lem:color_border}
Let $G = (V,E)$ be a connected signed graph with $E_+ \neq \emptyset$ and $E_- \neq \emptyset$.
Then we have $v \cedge{+} v' \cedge{-} v''$ for some $v$, $v'$, $v'' \in V$.
\end{lem}
\begin{proof}
By the assumption, vertex sets of an edge in $E_+$ and of an edge in $E_-$ are joined by a path.
This implies that $G$ involves a path $x_1 x_2 \cdots x_k$ with $x_1 \cedge{+} x_2$ and $x_{k-1} \cedge{-} x_k$.
Now this path must involve a desired triple.
\end{proof}

\section{Generalization of Elimination Orderings to Signed Graphs}
\label{sec:SEO}

\subsection{Definition}
\label{sec:SEO_definition}

As a generalization of perfect elimination orderings for non-signed graphs to signed graphs, here we introduce the following notion:
\begin{defn}
\label{defn:signed-elimination}
{\rm
Let $G = (V,E)$ be a signed graph and $\nu$ an ordering on $G$.
Then we say that $\nu$ is a {\em signed elimination ordering}, or a {\em SEO} in short, if for any triple $(u,v,w)$ of vertices of $G$ such that $\nu(u) < \nu(w) > \nu(v)$, and for each $\sigma \in \{+,-\}$, we have
\begin{description}
\item[(E1)] if $u \cedge{\sigma} w \cedge{\sigma} v$, then $u \cedge{\sigma} v$;
\item[(E2)] if $u \cedge{\sigma} v \cedge{-\sigma} w$, then $u \cedge{\sigma} w$.
\end{description}
We call the graph $G$ {\em signed-eliminable}, or {\em SE} in short, if a SEO on $G$ exists.
}
\end{defn}
In other words, when we assign weights $\omega(vv')$ to pairs $vv'$ of vertices $v,v'$ of $G$ by the rule that $\omega(vv') = \pm 1$ and $0$ if $vv' \in E_{\pm}$ and $vv' \not\in E$, respectively, it follows that SEOs are the orderings $\nu$ such that for any triple $(u,v,w)$ with $\nu(u) < \nu(w) > \nu(v)$, if $a \leq b \leq c$ are three weights $\omega(uv)$, $\omega(vw)$ and $\omega(uw)$ in nondecreasing order, then $b = \omega(uv)$ unless $u \noedge w \noedge v$.
This definition is motivated by recent research on hyperplane arrangements by Abe and Numata (see \cite{ANN}), that generalize Stanley's characterization \cite{Sta1} of certain \lq\lq free'' arrangements in terms of existence of a perfect elimination ordering on the corresponding graph.
\begin{exmp}
\label{exmp:SEgraph}
The signed graphs in Figure \ref{fig:SEgraphs} are signed-eliminable for any $n \geq 2$ (compare these with non-SE graphs given in Definition \ref{defn:forbidden_subgraphs}).
For the graph in the left, a SEO is given by $w \mapsto 1$ and $v_i \mapsto i+1$.
On the other hand, for the graph in the right, a SEO is given by $w_1 \mapsto 1$, $w_2 \mapsto 2$ and $v_i \mapsto i+2$ (which can be derived by the former result and Lemma \ref{lem:simplicial_expand} below).
\end{exmp}
\begin{figure}[htbp]
\centering
\begin{picture}(110,56)(0,-2)
\put(10,10){\circle{8}}\put(14,9){\line(1,0){12}}\put(14,11){\line(1,0){12}}
\put(10,-2){\hbox to0pt{\hss$v_1$\hss}}
\put(30,10){\circle{8}}\put(34,9){\line(1,0){9}}\put(34,11){\line(1,0){9}}
\put(30,-2){\hbox to0pt{\hss$v_2$\hss}}
\put(47,7){$\cdots$}\put(63,9){\line(1,0){9}}\put(63,11){\line(1,0){9}}
\put(76,10){\circle{8}}\put(80,9){\line(1,0){12}}\put(80,11){\line(1,0){12}}
\put(76,-2){\hbox to0pt{\hss$v_{n-1}$\hss}}
\put(96,10){\circle{8}}
\put(96,-2){\hbox to0pt{\hss$v_n$\hss}}
\put(53,36){\circle{8}}\put(50,33){\line(-1,-1){19}}\put(56,33){\line(1,-1){19}}
\put(52,32){\line(-1,-2){6}}\put(54,32){\line(1,-2){6}}
\put(53,44){\hbox to0pt{\hss$w$\hss}}
\put(56,34){\line(2,-1){39}}
\end{picture}
\quad
\begin{picture}(110,56)(0,-2)
\put(10,10){\circle{8}}\put(14,9){\line(1,0){12}}\put(14,11){\line(1,0){12}}
\put(10,-2){\hbox to0pt{\hss$v_1$\hss}}
\put(30,10){\circle{8}}\put(34,9){\line(1,0){9}}\put(34,11){\line(1,0){9}}
\put(30,-2){\hbox to0pt{\hss$v_2$\hss}}
\put(47,7){$\cdots$}\put(63,9){\line(1,0){9}}\put(63,11){\line(1,0){9}}
\put(76,10){\circle{8}}\put(80,9){\line(1,0){12}}\put(80,11){\line(1,0){12}}
\put(76,-2){\hbox to0pt{\hss$v_{n-1}$\hss}}
\put(96,10){\circle{8}}
\put(96,-2){\hbox to0pt{\hss$v_n$\hss}}
\put(41,36){\circle{8}}\put(65,36){\circle{8}}\put(45,36){\line(1,0){16}}
\put(41,44){\hbox to0pt{\hss$w_1$\hss}}\put(65,44){\hbox to0pt{\hss$w_2$\hss}}
\put(38,33){\line(-4,-3){26}}\put(39,32){\line(-1,-2){9}}\put(42,32){\line(1,-2){7}}\put(44,33){\line(3,-2){30}}
\put(61,34){\line(-4,-3){28}}\put(63,32){\line(-1,-2){7}}\put(66,32){\line(1,-2){9}}\put(68,33){\line(3,-2){28}}
\put(45,35){\line(2,-1){47}}
\end{picture}
\caption{Examples of SE graphs}
\label{fig:SEgraphs}
\end{figure}
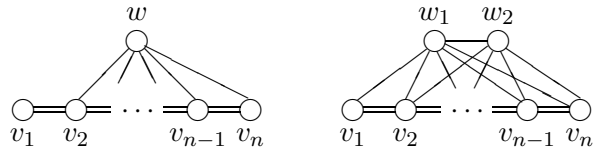
\begin{rem}
\label{rem:Gi_chordal}
{\rm
By condition (E1), a SEO on $G$ is also a perfect elimination ordering on both $G_+$ and $G_-$, thus Theorem \ref{thm:chordal} implies that $G_+$ and $G_-$ must be chordal if $G$ is SE.
In particular, when $E_+ = \emptyset$ or $E_- = \emptyset$, the SEOs on $G$ are precisely the perfect elimination orderings on $G$, therefore in this case $G$ is SE if and only if $G$ is chordal.
Thus SEOs are a generalization of the usual perfect elimination orderings.
}
\end{rem}
\begin{rem}
\label{rem:subgraph_bieliminable}
{\rm
The restriction of any SEO on a signed graph to its induced subgraph is also a SEO.
Thus the property of being SE is closed under taking induced subgraphs.
}
\end{rem}
\begin{rem}
\label{rem:component_bieliminable}
{\rm
A signed graph is SE if and only if every connected component of the graph is SE.
}
\end{rem}
The aim of this paper is to give a characterization of SE graphs.

\subsection{Fundamental Properties}
\label{sec:SEO_property}

In this subsection, we present fundamental properties of SE graphs for later references.
Let $G = (V,E)$ be a signed graph.
We start with the following observation:
\begin{lem}
\label{lem:SignedSimplicial}
Suppose that $\nu$ is a SEO on $G$, and $v = \nu^{-1}(|V|) \in V$.
Then the restriction $\nu|_{V \setminus v}$ of $\nu$ is a SEO on $G \setminus v$, and the following conditions hold:
\begin{description}
\item[(S1)] For each $\sigma \in \{+,-\}$, $N_{G_{\sigma}}\left[v\right]$ is a clique in $G_{\sigma}$ (that is, $v$ is simplicial in $G_{\sigma}$).
\item[(S2)] For each $\sigma \in \{+,-\}$, if $u \cedge{-\sigma} w \cedge{\sigma} v$, then $u \cedge{-\sigma} v$.
\end{description}
We call a vertex $v \in V$ {\em signed-simplicial} if it satisfies these two conditions.
\end{lem}
\begin{proof}
The first claim follows from Remark \ref{rem:subgraph_bieliminable}.
For the second claim, the first condition is satisfied by Lemma \ref{lem:ElimOrd_induction} and Remark \ref{rem:Gi_chordal}.
For the second condition, since $\nu(u) < \nu(v)$ and $\nu(w) < \nu(v)$, we have $u \cedge{-\sigma} v$ by condition (E2) in Definition \ref{defn:signed-elimination}.
Hence the claim holds.
\end{proof}
\begin{rem}
\label{rem:NecessaryCondition}
{\rm
If $v \in V$ is signed-simplicial, then $N_{\sigma}\left[v\right]$ is a {\em maximal} clique of $G_{\sigma}$ for each $\sigma \in \{+,-\}$.
On the other hand, if $v \in V$ and $N_G\left[v\right] = V$, then condition (S1) for $v$ implies condition (S2) for $v$.
}
\end{rem}
Owing to Lemma \ref{lem:SignedSimplicial}, we introduce the following notation:
\begin{defn}
\label{defn:M}
{\rm 
Let $S(G)$ denote the set of the signed-simplicial vertices of $G$; thus $S(G) \neq \emptyset$ if $G$ is signed-eliminable.
}
\end{defn}
\begin{rem}
\label{rem:M_nondecreasing}
{\rm
By the definition of $S(G)$, we have $v \in S(G|_{V'})$ if $v \in S(G)$ and $v \in V' \subset V$.
}
\end{rem}
Our next result shows that the \lq\lq converse'' of Lemma \ref{lem:SignedSimplicial} is also valid:
\begin{lem}
\label{lem:SEO_extension}
Suppose that $v \in S(G)$ and $\nu$ is a SEO on $G \setminus v$.
Then the unique extension $\overline{\nu}$ of $\nu$ to $V$ with $\overline{v}(v) = |V|$ is also a SEO on $G$.
Any SEO on $G$ is obtained in such a manner.
\end{lem}
\begin{proof}
First, note that the last claim is a restatement of Lemma \ref{lem:SignedSimplicial}.
To prove that $\overline{\nu}$ is a SEO, since $\nu$ is a SEO on $G \setminus v$, it suffices to show that conditions (E1) and (E2) are satisfied for $\overline{\nu}$ when $v$ plays the role of $w$ in these conditions.
Now (E1) and (E2) follow from the conditions (S1) and (S2), respectively, for $v$ to be signed-simplicial.
\end{proof}
\begin{cor}
\label{cor:GreedySucceed}
Suppose that $G$ is signed-eliminable.
Then for {\em any} $v \in S(G)$, there is a SEO $\nu$ on $G$ such that $\nu(v) = |V|$.
\end{cor}
\begin{proof}
Remark \ref{rem:subgraph_bieliminable} implies that $G \setminus v$ is SE, therefore a SEO $\nu$ on $G \setminus v$ exists.
This $\nu$ extends to the desired ordering on $G$ by Lemma \ref{lem:SEO_extension}.
\end{proof}

\subsection{An Algorithm to Find Signed Elimination Orderings}
\label{subsec:SEO_algorithm}

Summarizing the results in the previous sections, here we give a greedy algorithm which enables us to decide whether or not a given signed graph $G$ is signed-eliminable and to construct a SEO on $G$ (whenever it exists).
The next lemma is a key ingredient of our algorithm:
\begin{lem}
\label{lem:BiElimOrd_and_M}
Let $G = (V,E)$ be a signed graph.
\begin{enumerate}
\item Let $\nu$ be a SEO on $G$, and put $v_i = \nu^{-1}(i) \in V$ for $1 \leq i \leq |V|$.
Then
\begin{equation}
\label{eq:BiElimOrd_and_M}
v_i \in S(G|_{\{v_1,v_2,\dots,v_i\}}) \mbox{ for each } 1 \leq i \leq |V| \enspace.
\end{equation}
\item Conversely, let $V = \{v_1,v_2,\dots,v_n\}$ be a numbering of elements of $V$ satisfying the condition (\ref{eq:BiElimOrd_and_M}).
Then the map $\nu:v_i \mapsto i$ is a SEO on $G$.
\end{enumerate}
\end{lem}
\begin{proof}
Put $V_i = \{v_1,\dots,v_i\}$.
The former claim follows from Remark \ref{rem:subgraph_bieliminable} and Lemma \ref{lem:SignedSimplicial}; namely, $\nu|_{V_i}$ is a SEO on $G|_{V_i}$ for each $i$.
On the other hand, for the latter claim, it follows from Lemma \ref{lem:SEO_extension} and induction on $i$ that the restriction of $\nu$ to $V_i$ is a SEO on $G|_{V_i}$.
Thus the claim holds since $V_n = V$.
\end{proof}
Now our algorithm is described as follows:
\begin{thm}
\label{thm:GreedyAlgorithm}
Consider the following algorithm (with input $G$):
\begin{description}
\item[Step 1:] If $V = \emptyset$, then output an empty sequence $()$.
Otherwise, go to Step 2.
\item[Step 2:] Find a vertex $v \in S(G)$ (by, for example, checking the condition of being signed-simplicial for every vertex) and go to Step 3.
If such a vertex does not exist, output NULL.
\item[Step 3:] Perform this algorithm recursively for input $G \setminus v$.
If it outputs a sequence $(w_1,\dots,w_k)$, then output a sequence $(w_1,\dots,w_k,v)$.
If it outputs NULL, then output NULL.
\end{description}
Then $G$ is signed-eliminable if and only if the algorithm outputs a (possibly empty) sequence, not NULL.
Moreover, if the output is a sequence $(v_1,\dots,v_n)$, then the map $\nu:v_i \mapsto i$ is a SEO on $G$.
\end{thm}
\begin{proof}
First, if the algorithm outputs a sequence $(v_1,\dots,v_n)$, then this sequence satisfies the condition in Lemma \ref{lem:BiElimOrd_and_M}(\ref{eq:BiElimOrd_and_M}) by the construction, therefore $G$ is SE with a SEO $\nu$.
On the other hand, if $G$ is SE, then a $v \in S(G)$ is found in Step 2 by Lemma \ref{lem:SignedSimplicial}(2), while $G \setminus v$ is also SE by Remark \ref{rem:subgraph_bieliminable}.
Thus the output in Step 3 for input $G \setminus v$ is not NULL by induction on $|V|$, therefore the output for input $G$ is also not NULL.
Hence the proof is concluded.
\end{proof}

\subsection{Invariants for Signed-Eliminable Graphs}
\label{subsec:invariant}

In this subsection, we introduce the following object associated to each SE graph that can be computed from a given SEO, and prove that it is in fact independent of the choice of the SEO; therefore the object is an invariant for SE graphs.
The definition is the following:
\begin{defn}
\label{defn:degree_sequence}
Let $G = (V,E)$ be a signed-eliminable graph with $n$ vertices and $\nu$ a SEO on $G$.
Then for each $1 \leq i \leq n$, define a pair $\mathrm{d}^{(\nu)}(i) = (\mathrm{d}^{(\nu)}_+(i),\mathrm{d}^{(\nu)}_-(i))$ of nonnegative integers by
\begin{displaymath}
\mathrm{d}^{(\nu)}_{\sigma}(i) = \left|\{v \in V \mid \nu(v) \leq i \mbox{ and } v_i \cedge{\sigma} v\}\right| \mbox{ for each } \sigma \in \{+,-\} \enspace,
\end{displaymath}
where $v_i = \nu^{-1}(i) \in V$.
Moreover, let $\mathrm{d}^{(\nu)}$ denote the {\em multiset} consisting of all pairs $\mathrm{d}^{(\nu)}(i)$ with $1 \leq i \leq n$.
\end{defn}
For example, if $G$ is a graph with $v_1 \cedge{+} v_2 \cedge{+} v_4$ and $v_3 \cedge{-} v_4$ (and having no other vertices and no other edges) and $\nu$ is a SEO on $G$ such that $\nu(v_i) = i$, then $\mathrm{d}^{(\nu)} = \{\mathrm{d}^{(\nu)}(1),\dots,\mathrm{d}^{(\nu)}(4)\} = \{(0,0),(1,0),(0,0),(1,1)\}$.
For this object, we have the following property:
\begin{prop}
\label{prop:degree_sequence_invariant}
For any signed-eliminable graph $G$, the multiset $\mathrm{d}^{(\nu)}$ does not depend on the choice of a SEO $\nu$ on $G$.
Hence $\mathrm{d}^{(\nu)}$ gives an invariant for signed-eliminable graphs.
\end{prop}
\begin{proof}
Let $\nu$ and $\mu$ be two SEOs on the same $G$, and put $n = |V|$, $v_i = \nu^{-1}(i) \in V$, $w_i = \mu^{-1}(i) \in V$, $V_i = \{v_1,\dots,v_i\}$ and $G_i = G|_{V_i}$.
First, we show that $\mathrm{d}^{(\nu)} = \mathrm{d}^{(\mu)}$ (as multisets) if $v_{i-1} = w_i$ and $v_i = w_{i-1}$ for some $2 \leq i \leq n$ and $v_j = w_j$ for any $1 \leq j \leq n$ other than $i-1$ and $i$.
Now we have $\mathrm{d}^{(\nu)}(j) = \mathrm{d}^{(\mu)}(j)$ for any $1 \leq j \leq n$ other than $i-1$ and $i$ by definition, therefore it suffices to show that either $(\mathrm{d}^{(\nu)}(i-1),\mathrm{d}^{(\nu)}(i)) = (\mathrm{d}^{(\mu)}(i-1),\mathrm{d}^{(\mu)}(i))$ or $(\mathrm{d}^{(\nu)}(i-1),\mathrm{d}^{(\nu)}(i)) = (\mathrm{d}^{(\mu)}(i),\mathrm{d}^{(\mu)}(i-1))$ holds.
If $v_{i-1} \noedge v_i$, then it follows immediately from the definition of $\mathrm{d}^{(\nu)}$ that $\mathrm{d}^{(\nu)}(i-1) = \mathrm{d}^{(\mu)}(i)$ and $\mathrm{d}^{(\nu)}(i) = \mathrm{d}^{(\mu)}(i-1)$.
On the other hand, suppose that $v_{i-1} \cedge{\sigma} v_i$ for some $\sigma \in \{+,-\}$.
Put $X_v^{\tau} = \{v_j \mid 1 \leq j \leq i-2, v \cedge{\tau} v_j\}$ for $v \in \{v_{i-1},v_i\}$ and $\tau \in \{+,-\}$.
Now we have $X_{v_i}^{\sigma} \subset X_{v_{i-1}}^{\sigma}$ by the condition (E1) for $G_i$ and $\nu|_{V_i}$.
Similarly, we have $X_{v_{i-1}}^{-\sigma} \subset X_{v_i}^{-\sigma}$ by the condition (E2) for $G_i$ and $\nu|_{V_i}$.
Moreover, we also have $X_{v_{i-1}}^{\sigma} \subset X_{v_i}^{\sigma}$ and $X_{v_i}^{-\sigma} \subset X_{v_{i-1}}^{-\sigma}$ by exchanging the roles of $\nu$ and $\mu$ (recall that $w_{i-1} = v_i$ and $w_i = v_{i-1}$); thus $X_{v_{i-1}}^{\pm\sigma} = X_{v_i}^{\pm\sigma}$, respectively.
This implies that
\begin{eqnarray*}
&&\mathrm{d}^{(\nu)}_{\pm\sigma}(i-1) = |X_{v_{i-1}}^{\pm\sigma}| = |X_{w_{i-1}}^{\pm\sigma}| = \mathrm{d}^{(\mu)}_{\pm\sigma}(i-1) \enspace,\\
&&\mathrm{d}^{(\nu)}_{-\sigma}(i) = |X_{v_i}^{-\sigma}| = |X_{w_i}^{-\sigma}| = \mathrm{d}^{(\mu)}_{-\sigma}(i) \enspace,\\
&&\mathrm{d}^{(\nu)}_{\sigma}(i) = |X_{v_i}^{\sigma}| + 1 = |X_{w_i}^{\sigma}| + 1 = \mathrm{d}^{(\mu)}_{\sigma}(i) \enspace,
\end{eqnarray*}
therefore the claim of this paragraph follows.\\
\qquad
To conclude the proof, choose the index $i$ with $v_i = w_n$.
Now if $i < n$, then we have $w_n \in S(G)$ and $v_{i+1} \in S(G_{i+1})$ by Lemma \ref{lem:BiElimOrd_and_M}(1); therefore $v_i = w_n \in S(G_{i+1})$ and $v_{i+1} \in S(G_{i+1} \setminus v_i)$ by Remark \ref{rem:M_nondecreasing}.
By Lemma \ref{lem:BiElimOrd_and_M}, it follows that the ordering $\nu'$ on $G$ with $\nu'(v_i) = i+1$, $\nu'(v_{i+1}) = i$ and $\nu'(v_j) = j$ for any $1 \leq j \leq n$ other than $i$ and $i+1$ is also a SEO on $G$; therefore $\mathrm{d}^{(\nu)} = \mathrm{d}^{(\nu')}$ by the previous paragraph.
Iterating this process, we obtain a SEO $\nu''$ on $G$ such that $\mathrm{d}^{(\nu)} = \mathrm{d}^{(\nu'')}$ and $\nu''(w_n) = n$; while it follows from induction on $n$ that $\mathrm{d}^{(\nu''|_{V'})} = \mathrm{d}^{(\mu|_{V'})}$ where $V' = V \setminus w_n$, therefore $\mathrm{d}^{(\nu'')} = \mathrm{d}^{(\mu)}$.
Hence we have $\mathrm{d}^{(\nu)} = \mathrm{d}^{(\mu)}$, concluding the proof.
\end{proof}

In the special case of non-signed graphs, Proposition \ref{prop:degree_sequence_invariant} coincides with a result of Donald J.\ Rose \cite[Theorem 4]{Ros}.
Note that $\mathrm{d}^{(\nu)}(1) = (0,0)$ for any case.
This proposition implies that for any map $f$, the multiset consisting of $f(\mathrm{d}^{(\nu)}(i))$ for all $1 \leq i \leq |V|$ is also an invariant for SE graphs.
In particular, we have the following corollary, that plays a significant role in \cite{ANN}:
\begin{cor}
\label{cor:degree_invariant}
For a signed-eliminable graph $G$, define a multiset $\widetilde{\mathrm{deg}}(G)$ as consisting of the values $\mathrm{d}^{(\nu)}_{+}(i) - \mathrm{d}^{(\nu)}_{-}(i)$ for all $1 \leq i \leq |V|$, where $\nu$ is a SEO on $G$.
Then $\widetilde{\mathrm{deg}}(G)$ is independent of the SEO $\nu$; therefore it is an invariant for signed-eliminable graphs.
\end{cor}

\section{Lemmas for Characterization of Signed-Eliminable Graphs}
\label{sec:characterization_preliminary}

In this section, we prove that any signed graph with at most three vertices is SE, we present special examples of signed graphs that are not SE, and we give some further auxiliary properties.
Let $G = (V,E)$ denote a signed graph throughout this section.

First, we have the following:
\begin{prop}
\label{prop:ThreeVertices}
If $|V| \leq 3$, then $G$ is always signed-eliminable.
\end{prop}
\begin{proof}
This is trivial if $|V| \leq 2$.
For the case $|V| = 3$, Remark \ref{rem:Gi_chordal} implies that $G$ is SE if $E_+ = \emptyset$ or $E_- = \emptyset$.
On the other hand, if $E_+ \neq \emptyset$ and $E_- = \emptyset$, then $N_+(v) \neq \emptyset$ and $N_-(v) \neq \emptyset$ for some $v \in V$, and now we have $v \in S(G)$.
Thus Lemma \ref{lem:SEO_extension} and induction on $|V|$ imply that $G$ is SE.
Hence the proof is concluded.
\end{proof}

Secondly, we give the following observations which will be used in our argument several times:
\begin{lem}
\label{lem:simplicial_expand}
Suppose that $V$ is a disjoint union of $V'$ and $V''$, and $v \in S(G|_{V'})$.
Suppose further that $\sigma \in \{+,-\}$, $\left[v,V''\right] \subset E_{\sigma}$, $V''$ is a clique in $G_{\sigma}$, and $\left[V',V''\right] \cap E_{-\sigma} = \emptyset$.
Then $v \in S(G)$ if the following condition is satisfied:
\begin{description}
\item[(D)] If $v \cedge{\sigma} w \in V'$ and $w' \in V''$, then $w \cedge{\sigma} w'$.
\end{description}
\end{lem}
\begin{proof}
Put $G' = G|_{V'}$.
For condition (S1), the assumption implies that $N_{G_{-\sigma}}\left[v\right] = N_{G'_{-\sigma}}\left[v\right]$ is a clique in $G_{-\sigma}$, $N_{G_{\sigma}}\left[v\right] = N_{G'_{\sigma}}\left[v\right] \cup V''$, and the latter set is a clique in $G_{\sigma}$ by (D).
For condition (S2), suppose that $\tau \in \{+,-\}$ and $u \cedge{-\tau} w \cedge{\tau} v$.
It suffices to show that $u \cedge{-\tau} v$ when $\{u,w\} \not\subset V'$.
Now we have $w \in V'$; otherwise, $\tau = \sigma$ but $N_{G_{-\sigma}}(w) = \emptyset$, a contradiction.
This implies that $u \in V''$, $\tau = -\sigma$ (since $u \nocedge{-\sigma} w$) and $u \cedge{-\tau} v$.
Hence the claim holds.
\end{proof}
\begin{lem}
\label{lem:simplicial_expand_singleton}
Suppose that $v' \in V$ and $v \in S(G \setminus v')$.
\begin{enumerate}
\item Suppose further that $\sigma \in \{+,-\}$ and $v \cedge{\sigma} v'$.
Then $v \in S(G)$ if the following two conditions are satisfied:
\begin{description}
\item[(D'1)] If $v \cedge{\sigma} w \neq v'$, then $w \cedge{\sigma} v'$.
\item[(D'2)] If $v' \cedge{-\sigma} w$, then $v \edge w$.
\end{description}
\item Suppose further that $v \noedge v'$.
Then $v \in S(G)$ if the following condition is satisfied:
\begin{description}
\item[(D'')] $N_{G_{\tau}}(v) \cap N_{G_{-\tau}}(v') = \emptyset$ for each $\tau \in \{+,-\}$.
\end{description}
\end{enumerate}
\end{lem}
\begin{proof}
Put $G' = G \setminus v'$.
For the former claim, since $v \in S(G')$, condition (D'1) implies that $N_{G_{-\sigma}}\left[v\right] = N_{G'_{-\sigma}}\left[v\right]$ and $N_{G_{\sigma}}\left[v\right] = N_{G'_{\sigma}}\left[v\right] \cup v'$ are cliques in $G_{-\sigma}$ and $G_{\sigma}$, respectively.
Moreover, if $\tau \in \{+,-\}$ and $u \cedge{-\tau} v' \cedge{\tau} v$, then $\tau = \sigma$ and $u \cedge{-\sigma} v$ by (D'1) and (D'2).
If $v' \cedge{-\tau} u \cedge{\tau} v$, then $\tau \neq \sigma$ by (D'1), thus $\tau = -\sigma$ and $v' \cedge{-\tau} v$.
Since $v \in S(G')$, these imply that condition (S2) is satisfied.
Thus we have $v \in S(G)$.\\
\qquad
For the latter claim, we have $N_{G_+}\left[v\right] = N_{G'_+}\left[v\right]$ and $N_{G_-}\left[v\right] = N_{G'_-}\left[v\right]$, therefore condition (S1) is satisfied since $v \in S(G')$.
On the other hand, since $v \noedge v'$, (D'') implies that if $u \cedge{\sigma} w \cedge{\sigma} v$, then $u,w \neq v'$.
Since $v \in S(G')$, this implies that condition (S2) is also satisfied.
Thus we have $v \in S(G)$.
Hence the proof is concluded.
\end{proof}

Here we introduce the following special signed graphs that are not SE; these graphs will play a significant role in our characterization:
\begin{defn}
\label{defn:forbidden_subgraphs}
{\rm
\begin{enumerate}
\item We say that a sequence $(v_1,v_2,\dots,v_n;w)$ of vertices with $n \geq 3$ is a ({\em $\sigma$-}){\em mountain}, where $\sigma \in \{+,-\}$, if $v_i \cedge{-\sigma} v_{i+1}$ for $1 \leq i \leq n-1$, $w \cedge{\sigma} v_i$ for $2 \leq i \leq n-1$, and any other pair of vertices is not joined by an edge (see the left-hand side of Figure \ref{fig:forbidden_subgraphs}).
\item We say that a sequence $(v_1,v_2,\dots,v_n;w_1,w_2)$ of vertices with $n \geq 2$ is a ({\em $\sigma$-}){\em hill}, where $\sigma \in \{+,-\}$, if $v_i \cedge{-\sigma} v_{i+1}$ for $1 \leq i \leq n-1$, $w_1 \cedge{\sigma} w_2$, $w_1 \cedge{\sigma} v_i$ for $1 \leq i \leq n-1$, $w_2 \cedge{\sigma} v_i$ for $2 \leq i \leq n$, and any other pair of vertices is not joined by an edge (see the right-hand side of Figure \ref{fig:forbidden_subgraphs}).
\end{enumerate}
}
\end{defn}
\begin{figure}[htbp]
\centering
\begin{picture}(110,56)(0,-2)
\put(10,10){\circle{8}}\put(14,9){\line(1,0){12}}\put(14,11){\line(1,0){12}}
\put(10,-2){\hbox to0pt{\hss$v_1$\hss}}
\put(30,10){\circle{8}}\put(34,9){\line(1,0){9}}\put(34,11){\line(1,0){9}}
\put(30,-2){\hbox to0pt{\hss$v_2$\hss}}
\put(47,7){$\cdots$}\put(63,9){\line(1,0){9}}\put(63,11){\line(1,0){9}}
\put(76,10){\circle{8}}\put(80,9){\line(1,0){12}}\put(80,11){\line(1,0){12}}
\put(76,-2){\hbox to0pt{\hss$v_{n-1}$\hss}}
\put(96,10){\circle{8}}
\put(96,-2){\hbox to0pt{\hss$v_n$\hss}}
\put(53,36){\circle{8}}\put(50,33){\line(-1,-1){19}}\put(56,33){\line(1,-1){19}}
\put(52,32){\line(-1,-2){6}}\put(54,32){\line(1,-2){6}}
\put(53,44){\hbox to0pt{\hss$w$\hss}}
\end{picture}
\quad
\begin{picture}(110,56)(0,-2)
\put(10,10){\circle{8}}\put(14,9){\line(1,0){12}}\put(14,11){\line(1,0){12}}
\put(10,-2){\hbox to0pt{\hss$v_1$\hss}}
\put(30,10){\circle{8}}\put(34,9){\line(1,0){9}}\put(34,11){\line(1,0){9}}
\put(30,-2){\hbox to0pt{\hss$v_2$\hss}}
\put(47,7){$\cdots$}\put(63,9){\line(1,0){9}}\put(63,11){\line(1,0){9}}
\put(76,10){\circle{8}}\put(80,9){\line(1,0){12}}\put(80,11){\line(1,0){12}}
\put(76,-2){\hbox to0pt{\hss$v_{n-1}$\hss}}
\put(96,10){\circle{8}}
\put(96,-2){\hbox to0pt{\hss$v_n$\hss}}
\put(41,36){\circle{8}}\put(65,36){\circle{8}}\put(45,36){\line(1,0){16}}
\put(41,44){\hbox to0pt{\hss$w_1$\hss}}\put(65,44){\hbox to0pt{\hss$w_2$\hss}}
\put(38,33){\line(-4,-3){26}}\put(39,32){\line(-1,-2){9}}\put(42,32){\line(1,-2){7}}\put(44,33){\line(3,-2){30}}
\put(61,34){\line(-4,-3){28}}\put(63,32){\line(-1,-2){7}}\put(66,32){\line(1,-2){9}}\put(68,33){\line(3,-2){28}}
\end{picture}
\caption{Examples of non-SE graphs}
\label{fig:forbidden_subgraphs}
\end{figure}
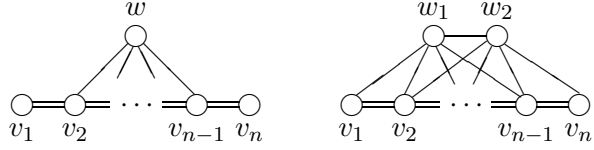
\begin{lem}
\label{lem:forbidden_subgraphs}
Any mountain and any hill are not signed-eliminable.
\end{lem}
\begin{proof}
Let $G$ be a $\sigma$-mountain or a $\sigma$-hill in Figure \ref{fig:forbidden_subgraphs} for $\sigma \in \{+,-\}$.
Then it suffices to show that $S(G) = \emptyset$ (see Definition \ref{defn:M}).
Now for the case of $\sigma$-mountain, $N_{-\sigma}\left[v_i\right]$ with $2 \leq i \leq n-1$ is not a clique in $G_{-\sigma}$, while none of $w$, $v_1$ and $v_n$ satisfies condition (S2) (focus on the subgraphs $v_1 \cedge{-\sigma} v_2 \cedge{\sigma} w$ and $v_n \cedge{-\sigma} v_{n-1} \cedge{\sigma} w$).
On the other hand, for the case of $\sigma$-hill, $N_{\sigma}\left[w_i\right]$ and $N_{-\sigma}\left[v_j\right]$ are not cliques in $G_{\sigma}$ and $G_{-\sigma}$, respectively, for $1 \leq i\ leq 2$ and $2 \leq j \leq n-1$, while neither $v_1$ nor $v_n$ satisfies condition (S2) (focus on the subgraphs $w_2 \cedge{\sigma} v_2 \cedge{-\sigma} v_1$ and $w_1 \cedge{\sigma} v_{n-1} \cedge{-\sigma} v_n$).
Thus we have $S(G) = \emptyset$ in both cases.
\end{proof}

Moreover, we present a key lemma in our argument:
\begin{lem}
\label{lem:AlternatingPath}
Suppose that $G$ is signed-eliminable and $u$, $v$, $w$ and $x$ are distinct vertices of $G$.
If $\sigma \in \{+,-\}$ and $u \cedge{\sigma} v \cedge{-\sigma} w \cedge{\sigma} x$, then $u \cedge{\sigma} x$, and we have either $u \cedge{\sigma} w$ or $v \cedge{\sigma} x$.
\end{lem}
\begin{proof}
By Remark \ref{rem:subgraph_bieliminable} and symmetry, we may assume without loss of generality that $V = \{u,v,w,x\}$, and $u \in S(G)$ or $v \in S(G)$.
If $u \in S(G)$, then we have $u \cedge{-\sigma} w$ by (S2) since $w \cedge{-\sigma} v \cedge{\sigma} u$, we have $u \cedge{\sigma} x$ by (S2) since $x \cedge{\sigma} w \cedge{-\sigma} u$, and we have $v \cedge{\sigma} x$ by (S1) since $v,x \in N_{\sigma}(u)$.
On the other hand, if $v \in S(G)$, then we have $v \cedge{\sigma} x$ by (S2) since $x \cedge{\sigma} w \cedge{-\sigma} v$, and we have $u \cedge{\sigma} x$ by (S1) since $u,x \in N_{\sigma}(v)$.
Thus the claim holds in any case.
\end{proof}

\section{A Full Characterization of Signed-Eliminable Graphs}
\label{sec:characterization}

In this section, we state and prove a full characterization of SE graphs, which is the main contribution of this paper.

\subsection{The Statement}
\label{subsec:characterization_statement}

Before giving our characterization, we introduce the following terminology:
We call an induced path in $G$ of the form $u \cedge{\sigma} v \cedge{-\sigma} w \cedge{\sigma} x$, where $\sigma \in \{+,-\}$, an {\em alternating $4$-path}.
Then our full characterization is described as the following theorem:
\begin{thm}
\label{thm:characterization}
Let $G$ be a signed graph.
Then $G$ is signed-eliminable if and only if all of the following three conditions are satisfied:
\begin{description}
\item[(C1)] Both $G_+$ and $G_-$ are chordal.
\item[(C2)] For any alternating $4$-path $u \cedge{\sigma} v \cedge{-\sigma} w \cedge{\sigma} x$ in $G$ (see above for terminology), we have either $w \cedge{\sigma} u \cedge{\sigma} x$ or $u \cedge{\sigma} x \cedge{\sigma} v$.
\item[(C3)] $G$ contains no mountain and no hill as an induced subgraph (see Definition \ref{defn:forbidden_subgraphs} for terminology).
\end{description}
\end{thm}
The \lq\lq only if'' part of Theorem \ref{thm:characterization} follows from Remark \ref{rem:Gi_chordal}, Lemma \ref{lem:AlternatingPath}, Remark \ref{rem:subgraph_bieliminable} and Lemma \ref{lem:forbidden_subgraphs}.
In the rest of this section, we prove the \lq\lq if'' part; that is, $G$ is SE if the conditions (C1)--(C3) are satisfied.
\begin{rem}
\label{rem:characterization_equivalent}
{\rm 
In a previous version of this paper, the characterization was stated in the following form: A signed graph $G$ is signed-eliminable if and only if (C1) and (C3) are satisfied and any induced subgraph of $G$ with four vertices is signed-eliminable.
This characterization is also valid by Theorem \ref{thm:characterization}, Remark \ref{rem:subgraph_bieliminable} and Lemma \ref{lem:AlternatingPath}.
(Note that we do not use this fact in our proof of Theorem \ref{thm:characterization}.)
}
\end{rem}

\subsection{Some Lemmas}
\label{subsec:proof_lemmas}

This subsection is devoted to present the following lemmas that will be used in our proof of the main theorem:
\begin{lem}
\label{lem:induced_path}
Suppose that the conditions (C1) and (C2) in Theorem \ref{thm:characterization} are satisfied.
If $\sigma \in \{+,-\}$, $k \geq 2$, $x_1x_2 \cdots x_k$ is an induced path in $G_{\sigma}$, $x_1 \cedge{\sigma} x_0 \neq x_2$ and $x_0 \noedge x_2$, then $x_0x_1 \cdots x_k$ is also an induced path in $G$.
\end{lem}
\begin{proof}
We proceed the proof by induction on $k$.
The case $k = 2$ is trivial, therefore suppose that $k \geq 3$ and $x_0x_1 \cdots x_{k-1}$ is an induced path in $G$.
Note that $x_0 \neq x_i$ for any $1 \leq i \leq k$ by the assumption.
Now if $x_k \cedge{-\sigma} x_i$ for some $1 \leq i \leq k-2$, then we have $x_{i-1} \cedge{\sigma} x_i \cedge{-\sigma} x_k \cedge{\sigma} x_{k-1}$ and $x_{i-1} \nocedge{\sigma} x_{k-1}$, contradicting (C2).
If $x_k \cedge{-\sigma} x_0$, then we have $x_{k-1} \cedge{\sigma} x_k \cedge{-\sigma} x_0 \cedge{\sigma} x_1$, while $x_{k-1} \nocedge{\sigma} x_0$ and $x_k \nocedge{\sigma} x_1$ by the assumption and the induction hypothesis.
This contradicts (C2).
Moreover, if $x_k \cedge{\sigma} x_0$, then $x_0x_1 \cdots x_kx_0$ is a cycle in $G_{\sigma}$ with at least four vertices, while this cycle has no chord since both $x_0x_1 \cdots x_{k-1}$ and $x_1x_2 \cdots x_k$ are induced paths in $G_{\sigma}$ by the assumption and the induction hypothesis.
This contradicts (C1).
Hence $x_0x_1 \cdots x_k$ is also an induced path in $G$, concluding the proof.
\end{proof}
Here we introduce the following notations.
For subsets $V' \subset V''$ of $V$ and $\sigma \in \{+,-\}$, we define $\mathrm{cl}_{\sigma}(V';V'')$ to be the union of vertex sets of the connected components of $G_{\sigma}|_{V''}$ that have nonempty intersection with $V'$, and define
\begin{eqnarray*}
\overline{\mathrm{cl}}_{\sigma}(V';V'')
&{}={}& \mathrm{cl}_{\sigma}(V';V'') \cup \{v \in V'' \mid N_{-\sigma}(v) \cap \mathrm{cl}_{\sigma}(V';V'') \neq \emptyset \} \enspace,\\
\partial_{\sigma}(V';V'')
&{}={}& \overline{\mathrm{cl}}_{\sigma}(V';V'') \setminus \mathrm{cl}_{\sigma}(V';V'') \enspace.
\end{eqnarray*}
\begin{lem}
\label{lem:simplicial_interior}
Let $V' \subset V'' \subset V$, $\sigma \in \{+,-\}$ and put $W = \mathrm{cl}_{\sigma}(V';V'')$ and $\overline{W} = \overline{\mathrm{cl}}_{\sigma}(V';V'')$.
Suppose that the condition (C2) is satisfied and every connected component of $G_{\sigma}|_W$ contains at least two vertices.
Then $S(G|_{\overline{W}}) \subset W$ and $S(G|_{\overline{W}}) \subset S(G|_{V''})$.
\end{lem}
\begin{proof}
Let $v \in S(G|_{\overline{W}})$.
First, to prove that $v \in W$, it suffices to consider the case that $N_{-\sigma}(v) \cap W \neq \emptyset$.
Then we have $v \cedge{-\sigma} u$ for some $u \in W$, while $u \cedge{\sigma} w$ for some $w \in W$ by the assumption.
Now condition (S2) implies that $v \cedge{\sigma} w$, therefore $v \in W$ since $v \in V''$ and $w \in W$.\\
\quad
From now, we show that $v \in S(G|_{V''})$.
Put $G' = G|_{\overline{W}}$ and $G'' = G|_{V''}$.
Now since $v \in W$, the definition of $\overline{W}$ implies that $N_{G''_+}\left[v\right] \subset W$ and $N_{G''_-}\left[v\right] \subset \overline{W}$, therefore condition (S1) holds since $v \in S(G')$.
Similarly, if $u,w \in V''$ and $u \cedge{-\sigma} w \cedge{\sigma} v$, then $w \in W$ and $u \in \overline{W}$, therefore $u \cedge{-\sigma} v$ by the condition (S2) for $v$ and $G'$.
Finally, suppose that $u,w \in V''$ and $u \cedge{\sigma} w \cedge{-\sigma} v$.
Then we have $w \in \overline{W}$ as above, while by the assumption, we have $v \cedge{\sigma} x$ for some $x \in W$.
Now we have $u \cedge{\sigma} v$ if $x = u$; thus suppose that $x \neq u$.
If $w \in W$, then we have $u \in W$ as above, therefore $u \cedge{\sigma} v$ by the condition (S2) for $v$ and $G'$.
On the other hand, if $w \in \overline{W} \setminus W$, then we have $w \nocedge{\sigma} x$, therefore (C2) implies that $x \cedge{\sigma} u \cedge{\sigma} v$ (since $u \cedge{\sigma} w \cedge{-\sigma} v \cedge{\sigma} x$ is an alternating $4$-path).
Thus $u \cedge{\sigma} v$ in any case, therefore condition (S2) holds.
Hence the proof is concluded.
\end{proof}
\begin{lem}
\label{lem:SE_simplicial_in_border}
Let $G = (V,E)$ be a connected signed-eliminable graph such that $E_+ \neq \emptyset$ and $E_- \neq \emptyset$.
Then there exists a vertex $v \in S(G)$ such that $N_+(v) \neq \emptyset$ and $N_-(v) \neq \emptyset$.
\end{lem}
\begin{proof}
Note that conditions (C1)--(C3) hold by the \lq\lq only if'' part of Theorem \ref{thm:characterization} (that has been proved in Section \ref{subsec:characterization_statement}).
By Lemma \ref{lem:color_border}, we have $u \cedge{+} v \cedge{-} w$ for some vertices $u$, $v$ and $w$ of $G$.
Now put $W = \mathrm{cl}_+(v;V)$ and $\overline{W} = \overline{\mathrm{cl}}_+(v,V)$.
Then $(G|_W)_+$ is connected and contains $u$ and $v$, therefore $S(G|_{\overline{W}}) \subset W$ and $S(G|_{\overline{W}}) \subset S(G)$ by Lemma \ref{lem:simplicial_interior}.
On the other hand, now $G|_{\overline{W}}$ is connected and contains $w$.
Thus by putting $X = \mathrm{cl}_-(v;\overline{W})$ and $\overline{X} = \overline{\mathrm{cl}}_-(v;\overline{W})$, it follows that $(G|_X)_-$ is connected and contains $v$ and $w$, therefore $S(G|_{\overline{X}}) \subset X$ and $S(G|_{\overline{X}}) \subset S(G|_{\overline{W}})$ by Lemma \ref{lem:simplicial_interior}.
Moreover, a vertex $x \in S(G|_{\overline{X}})$ exists by Remark \ref{rem:subgraph_bieliminable}.
Summarizing, we have $x \in S(G)$, $x \in W \cap X$, and both $(G|_W)_+$ and $(G|_X)_-$ are connected and contain at least two vertices, therefore $N_+(x) \neq \emptyset$ and $N_-(x) \neq \emptyset$.
Thus the proof is concluded.
\end{proof}

\subsection{Proof of Theorem \ref{thm:characterization}, First Step}
\label{subsec:proof_main_theorem_1}

In Sections \ref{subsec:proof_main_theorem_1}--\ref{subsec:proof_main_theorem_3}, we give a proof of the \lq\lq if'' part of Theorem \ref{thm:characterization}, namely we show that any signed graph $G = (V,E)$ satisfying the conditions (C1)--(C3) is signed-eliminable.
Since the conditions (C1)--(C3) are closed under taking induced subgraphs, we proceed the proof by induction on $|V|$.
By Proposition \ref{prop:ThreeVertices}, the claim is trivial if $|V| \leq 3$; thus suppose that $|V| \geq 4$.
Moreover, owing to Remarks \ref{rem:Gi_chordal} and \ref{rem:component_bieliminable}, the claim follows if either $G$ is not connected, or $E_+ = \emptyset$ or $E_- = \emptyset$.
Thus we may assume further that $G$ is connected, $E_+ \neq \emptyset$ and $E_- \neq \emptyset$, therefore we have $v \cedge{+} v' \cedge{-} v''$ for some vertices $v$, $v'$ and $v''$ of $G$ by Lemma \ref{lem:color_border}.
Now by Lemma \ref{lem:SEO_extension}, it suffices to show that $S(G) \neq \emptyset$.

In this subsection, we consider the case that
\begin{equation}
\label{eq:proof_main_condition_1}
\mbox{ if } u,u',u'' \in V \mbox{ and } u \cedge{+} u' \cedge{-} u'', \mbox{ then } u \edge u'' \enspace,
\end{equation}
and prove that $S(G) \neq \emptyset$ if condition (\ref{eq:proof_main_condition_1}) is satisfied.
\begin{lem}
\label{lem:proof_main_1}
In the above setting, there exists a vertex $w \in V$ such that either $N_{G_+}(w) = \emptyset$ or $N_{G_-}(w) = \emptyset$.
\end{lem}
\begin{proof}
Take a pair of a sequence $(w_1,w_2,\dots,w_k)$ of vertices of $G$ and a sequence $(\sigma_2,\sigma_3,\dots,\sigma_k)$ of signs $\sigma_i \in \{+,-\}$, with $k$ maximal, such that $\sigma_i = -\sigma_{i-1}$ for any $3 \leq i \leq k$ and $w_i w_j \in E_{\sigma_j}$ for any $1 \leq i < j \leq k$.
Note that $k \geq 3$, since the condition (\ref{eq:proof_main_condition_1}) implies that either $v \cedge{+} v''$ (now the pair of $(v'',v',v)$ and $(-,+)$ satisfies the condition) or $v \cedge{-} v''$ (now the pair of $(v,v',v'')$ and $(+,-)$ satisfies the condition).
We show that $N_{-\sigma_k}(w_k) = \emptyset$.
Assume contrary that $w_k x \in E_{-\sigma_k}$ for some $x \in V$.
Note that $x \neq w_i$ for any $1 \leq i \leq k$.
Now for each $1 \leq i \leq k-2$, we have $x w_k \in E_{-\sigma_k}$, $w_k w_{k-1} \in E_{\sigma_k}$ and $w_{k-1} w_i \in E_{-\sigma_k}$ (note that $\sigma_k = -\sigma_{k-1}$), therefore $x w_i \in E_{-\sigma_k}$ by condition (C2).
Moreover, we have $x w_k \in E_{-\sigma_k}$, $w_k w_{k-2} \in E_{\sigma_k}$ and $w_{k-2} w_{k-1} \in E_{-\sigma_k}$, therefore $x w_{k-1} \in E_{-\sigma_k}$ by condition (C2).
Thus we have $x w_i \in E_{-\sigma_k}$ for any $1 \leq i \leq k$, therefore the pair of $(w_1,\dots,w_k,x)$ and $(\sigma_2,\dots,\sigma_k,-\sigma_k)$ also satisfies the condition.
This contradicts the maximality of $k$.
Hence we have $N_{-\sigma_k}(w_k) = \emptyset$, therefore the claim holds.
\end{proof}
Owing to Lemma \ref{lem:proof_main_1}, we have $N_{G_{\sigma}}(w) = \emptyset$ for some $w \in V$ and $\sigma \in \{+,-\}$.
Since $E_{\sigma} \neq \emptyset$, we have $w' \cedge{\sigma} w''$ for some $w',w'' \in V \setminus w$.
Now put $W = \mathrm{cl}_{\sigma}(\{w',w''\};V \setminus w)$ and $\overline{W} = \overline{\mathrm{cl}}_{\sigma}(\{w',w''\};V \setminus w)$.
Then $(G|_W)_{\sigma}$ is connected and contains $w'$ and $w''$, while $S(G|_{\overline{W}}) \neq \emptyset$ by the induction hypothesis.
Thus Lemma \ref{lem:simplicial_interior} implies that $x \in S(G \setminus w)$ for some $x \in W$.\\
\quad
We show that $x \in S(G)$ by using Lemma \ref{lem:simplicial_expand_singleton}, where $x$ and $w$ play the roles of $v$ and $v'$ in that lemma, respectively.
If $x \noedge w$, then condition (D'') follows from condition (\ref{eq:proof_main_condition_1}).
On the other hand, if $x \cedge{-\sigma} w$, then condition (D'2) (where $-\sigma$ plays the role of $\sigma$) holds since $N_{\sigma}(w) = \emptyset$.
For condition (D'1), suppose that $x \cedge{-\sigma} y \neq w$.
Then, since $(G|_W)_{\sigma}$ is connected and contains at least two vertices, we have $x \cedge{\sigma} z$ for some $z \in W$.
Now since $N_{\sigma}(w) = \emptyset$ and $w \cedge{-\sigma} x \cedge{\sigma} z$, we have $w \cedge{-\sigma} z$ by condition (\ref{eq:proof_main_condition_1}).
Moreover, since $w \cedge{-\sigma} z \cedge{\sigma} x \cedge{-\sigma} y$, we have $w \cedge{-\sigma} y$ by (C2).
Thus condition (D'1) is also satisfied.
Hence we have $x \in S(G)$ by Lemma \ref{lem:simplicial_expand_singleton}, as desired.

\subsection{Proof of Theorem \ref{thm:characterization}, Second Step}
\label{subsec:proof_main_theorem_2}

From now, we consider the case that condition (\ref{eq:proof_main_condition_1}) does not hold, thus we have $v_+ \noedge v_-$ and $W = N_+(v_+) \cap N_-(v_-) \neq \emptyset$ for some vertices $v_+$ and $v_-$ of $G$.
Now for each $\sigma \in \{+,-\}$, put
\begin{eqnarray*}
X_{\sigma} &{}={}& \mathrm{cl}_{-\sigma}\bigl( W \cup v_{-\sigma} ; W \cup (V \setminus N(v_{-\sigma})) \bigr) \setminus (W \cup v_{-\sigma}) \enspace, \\
Y_{\sigma} &{}={}& \partial_{-\sigma}\bigl( W \cup v_{-\sigma} ; W \cup (V \setminus N(v_{-\sigma})) \bigr) \enspace,
\end{eqnarray*}
and put
\begin{displaymath}
V' = W \cup \{v_+,v_-\} \cup X_+ \cup Y_+ \cup X_- \cup Y_- \enspace.
\end{displaymath}
By the construction, $(G|_{X_{\sigma} \cup W \cup v_{-\sigma}})_{-\sigma}$ is connected and contains $W \cup v_{-\sigma}$ for each $\sigma \in \{+,-\}$.
Thus Remark \ref{rem:M_nondecreasing} and Lemma \ref{lem:simplicial_interior} imply that
\begin{displaymath}
S(G|_{V'}) \subset S(G|_{W \cup v_{-} \cup X_{+} \cup Y_{+}}) \cup S(G|_{W \cup v_{+} \cup X_{-} \cup Y_{-}}) \subset W \cup \{v_+,v_-\} \cup X_+ \cup X_- \enspace,
\end{displaymath}
while $v_+,v_- \not\in S(G|_{V'})$ by the choice of $v_+$ and $v_-$, therefore
\begin{equation}
\label{eq:proof_main_property_simplicial}
S(G|_{V'}) \subset W \cup X_+ \cup X_- \enspace.
\end{equation}
On the other hand, the construction implies that, for each $\sigma \in \{+,-\}$,
\begin{equation}
\label{eq:proof_main_property_edges}
\left[X_{\sigma} \cup W,Y_{\sigma}\right] \cap E_{-\sigma} = \emptyset \mbox{ and } (X_{\sigma} \cup Y_{\sigma}) \cap N(v_{-\sigma}) = \emptyset \enspace.
\end{equation}
Moreover, we have the following results:
\begin{lem}
\label{lem:proof_main_step2_1}
In the above setting, $X_{\sigma} \cup Y_{\sigma} \subset N_{\sigma}(v_{\sigma})$ for each $\sigma \in \{+,-\}$.
\end{lem}
\begin{proof}
First, we show that $u \nocedge{-\sigma} v_{\sigma}$ for any $u \in X_{\sigma} \cup Y_{\sigma}$.
If $u \cedge{-\sigma} v_{\sigma}$, then we have $u \cedge{-\sigma} v_{\sigma} \cedge{\sigma} w \cedge{-\sigma} v_{-\sigma}$ for any $w \in W$, therefore $u \cedge{-\sigma} v_{-\sigma}$ by (C2).
This contradicts (\ref{eq:proof_main_property_edges}).
Thus we have $u \nocedge{-\sigma} v_{\sigma}$.\\
\quad
Now it suffices to show that $u \edge v_{\sigma}$ for any $u \in X_{\sigma} \cup Y_{\sigma}$.
First, for the case $u \in X_{\sigma}$, we take an induced path $u_0u_1 \cdots u_k$ in $G_{-\sigma}$ such that $u_0 \in W$, $u_i \in X_{\sigma}$ for $1 \leq i \leq k-1$ and $u_k = u$ (such a path exists by the construction of $X_{\sigma}$ and (\ref{eq:proof_main_property_edges})), and prove that $v_{\sigma} \edge u$ by induction on $k$.
We have $v_{\sigma} \cedge{\sigma} u_i$ for $0 \leq i \leq k-1$ by induction hypothesis and the previous paragraph.
Now by (\ref{eq:proof_main_property_edges}) and Lemma \ref{lem:induced_path}, the path $v_{-\sigma} u_0u_1 \cdots u_k$ in $G_{-\sigma}$ is an induced path in $G$, therefore $(v_{-\sigma},u_0,\dots,u_k;v_{\sigma})$ is a $\sigma$-mountain if $v_{\sigma} \noedge u$.
Thus (C3) implies that $v_{\sigma} \edge u$.
Secondly, for the case $u \in Y_{\sigma}$, by the construction and (\ref{eq:proof_main_property_edges}), we have $u \cedge{\sigma} w$ for some $w \in W \cup X_{\sigma}$.
Now if $w \in X_{\sigma}$, then $w \cedge{-\sigma} x$ for some $x \in W \cup X_{\sigma}$, therefore $v_{\sigma} \cedge{\sigma} x \cedge{-\sigma} w \cedge{\sigma} u$ by the previous paragraph.
Thus (C2) implies that $v_{\sigma} \cedge{\sigma} u$.
On the other hand, suppose that $w \in W$.
Then (\ref{eq:proof_main_property_edges}) implies that $(v_{\sigma},w,u;v_{-\sigma})$ is a $(-\sigma)$-mountain if $v_{\sigma} \noedge u$, contradicting (C3).
Thus we have $v_{\sigma} \edge u$.
Hence the proof is concluded.
\end{proof}
\begin{lem}
\label{lem:proof_main_step2_2}
In the above setting, if $\sigma \in \{+,-\}$ and $u \in N_{-\sigma}(v_{\sigma})$, then $N_{-\sigma}(v_{\sigma}) \cup W \cup v_{-\sigma} \cup X_{\sigma} \cup X_{-\sigma} \cup Y_{-\sigma} \subset N_{-\sigma}\left[u\right]$.
\end{lem}
\begin{proof}
First, suppose that $w \in W$.
Then $u \cedge{-\sigma} v_{\sigma} \cedge{\sigma} w \cedge{-\sigma} v_{-\sigma}$, therefore $v_{-\sigma} \cedge{-\sigma} u \cedge{-\sigma} w$ by (C2) since $v_{\sigma} \noedge v_{-\sigma}$.
Moreover, if $u \neq u' \in N_{-\sigma}(v_{\sigma})$, then we have $u' \cedge{-\sigma} v_{\sigma} \cedge{\sigma} w \cedge{-\sigma} u$, therefore $u' \cedge{-\sigma} u$ by (C2).
These imply that $v_{-\sigma} \cup W \cup N_{-\sigma}(v_{\sigma}) \subset N_{-\sigma}\left[u\right]$.
On the other hand, if $x \in X_{\sigma}$, then we have $x' \cedge{-\sigma} x$ for some $x' \in W \cup X_{\sigma}$.
Now by Lemma \ref{lem:proof_main_step2_1}, we have $x \cedge{-\sigma} x' \cedge{\sigma} v_{\sigma} \cedge{-\sigma} u$, therefore $x \cedge{-\sigma} u$ by (C2).
Thus we have $X_{\sigma} \subset N_{-\sigma}\left[u\right]$.\\
\quad
From now, we show that $x \cedge{-\sigma} u$ for any $x \in X_{-\sigma} \cup Y_{-\sigma}$.
First, if $x \cedge{\sigma} u$, then $v_{\sigma} \cedge{-\sigma} u \cedge{\sigma} x \cedge{-\sigma} v_{-\sigma}$ by Lemma \ref{lem:proof_main_step2_1}, while $v_{\sigma} \noedge v_{-\sigma}$.
This contradicts (C2), therefore we have $x \nocedge{\sigma} u$.
Secondly, in the case $x \in X_{-\sigma}$, we take an induced path $x_0x_1 \cdots x_k$ in $G_{\sigma}$ such that $x_0 \in W$, $x_i \in X_{-\sigma}$ for $1 \leq i \leq k-1$ and $x_k = x$ (it exists by construction of $X_{-\sigma}$ and (\ref{eq:proof_main_property_edges})), and we prove $x \cedge{-\sigma} u$ by induction on $k$.
We have $x_i \cedge{-\sigma} u$ by induction hypothesis, while the path $v_{\sigma} x_0x_1 \cdots x_k$ in $G_{\sigma}$ is an induced path in $G$ by (\ref{eq:proof_main_property_edges}) and Lemma \ref{lem:induced_path}.
Now if $x \nocedge{-\sigma} u$, then $x \noedge u$ by the above result, therefore $(v_{\sigma},x_0,\dots,x_k;u,v_{-\sigma})$ is a $(-\sigma)$-hill.
This contradicts (C3), therefore we have $x \cedge{-\sigma} u$.
Finally, in the case $x \in Y_{-\sigma}$, we have $x \cedge{-\sigma} x'$ for some $x' \in X_{-\sigma} \cup W$, and $x' \cedge{\sigma} x''$ for some $x'' \in v_{\sigma} X_{-\sigma} \cup W \cup v_{\sigma}$.
Now we have $x \cedge{-\sigma} x' \cedge{\sigma} x'' \cedge{-\sigma} u$ by the above result, therefore $x \cedge{-\sigma} u$ by (C2).
Hence we have $X_{-\sigma} \cup Y_{-\sigma} \subset N_{-\sigma}\left[u\right]$, therefore the proof is concluded.
\end{proof}
\begin{lem}
\label{lem:proof_main_step2_3}
In the above setting, if $w \in W$, then $N_{\sigma}(w) \subset W \cup v_{\sigma} \cup X_{\sigma} \cup Y_{\sigma} \cup X_{-\sigma} \cup N_{\sigma}(v_{-\sigma})$ for each $\sigma \in \{+,-\}$.
\end{lem}
\begin{proof}
Let $u \in N_{\sigma}(w)$.
It suffices to consider the case that $u \neq v_{\sigma}$ and $u \nocedge{\sigma} v_{-\sigma}$.
Note that $u \nocedge{-\sigma} v_{\sigma}$ by Lemma \ref{lem:proof_main_step2_2}.
Now we have $u \in X_{-\sigma}$ if $u \noedge v_{\sigma}$, while $u \in X_{\sigma} \cup Y_{\sigma}$ if $u \noedge v_{-\sigma}$.
Moreover, we have $u \in W$ if $u \cedge{\sigma} v_{\sigma}$ and $u \cedge{-\sigma} v_{-\sigma}$.
Hence the claim holds in any case.
\end{proof}
\begin{lem}
\label{lem:proof_main_step2_4}
In the above setting, if $\sigma \in \{+,-\}$ and $x \in X_{\sigma}$, then $N_{\sigma}(x) \subset W \cup v_{\sigma} \cup X_{\sigma} \cup Y_{\sigma} \cup N_{\sigma}(v_{-\sigma})$ and $N_{-\sigma}(x) \subset W \cup X_{\sigma} \cup X_{-\sigma} \cup Y_{-\sigma} \cup N_{-\sigma}(v_{\sigma})$.
\end{lem}
\begin{proof}
First, we show that $u \in W \cup v_{\sigma} \cup X_{\sigma} \cup Y_{\sigma} \cup N_{\sigma}(v_{-\sigma})$ for any $u \in N_{\sigma}(x)$.
It suffices to consider the case that $u \neq v_{\sigma}$ and $u \nocedge{\sigma} v_{-\sigma}$.
Now by the choice of $x$, we have $x \cedge{-\sigma} x'$ for some $x' \in X_{\sigma} \cup W$, therefore $u \cedge{\sigma} x \cedge{-\sigma} x' \cedge{\sigma} v_{\sigma}$ by Lemma \ref{lem:proof_main_step2_1}.
Thus $u \cedge{\sigma} v_{\sigma}$ by (C2).
This implies that $u \in W$ if $u \cedge{-\sigma} v_{-\sigma}$, while $u \in X_{\sigma} \cup Y_{\sigma}$ if $u \noedge v_{-\sigma}$ (since $u \cedge{\sigma} x \in X_{\sigma}$).
Thus the claim for $N_{\sigma}(x)$ holds (since $u \neq v_{-\sigma}$).\\
\quad
Secondly, we show that $u \in W \cup X_{\sigma} \cup X_{-\sigma} \cup Y_{-\sigma} \cup N_{-\sigma}(v_{\sigma})$ for any $u \in N_{-\sigma}(x)$.
Note that $u \neq v_{\sigma}$ and $u \neq v_{-\sigma}$ by (\ref{eq:proof_main_property_edges}) and Lemma \ref{lem:proof_main_step2_1}, and $u \nocedge{\sigma} v_{-\sigma}$ by Lemma \ref{lem:proof_main_step2_2} (since $x \nocedge{\sigma} u$).
It suffices to consider the case that $u \nocedge{-\sigma} v_{\sigma}$.
Now we have $u \in X_{\sigma}$ if $u \noedge v_{-\sigma}$, while $u \in W$ if $u \cedge{\sigma} v_{\sigma}$ and $u \cedge{-\sigma} v_{-\sigma}$.
Finally, suppose that $u \noedge v_{\sigma}$ and $u \cedge{-\sigma} v_{-\sigma}$.
Take an induced path $x_0x_1 \cdots x_k$ in $G_{-\sigma}$ such that $x_0 \in W$, $x_i \in X_{\sigma}$ for $1 \leq i \leq k-1$ and $x_k = x$ (it exists by construction of $X_{\sigma}$ and (\ref{eq:proof_main_property_edges})).
Then (\ref{eq:proof_main_property_edges}) and Lemma \ref{lem:induced_path} imply that the path $v_{-\sigma} x_0x_1 \cdots x_k$ is an induced path in $G_{-\sigma}$, while $v_{-\sigma} \cedge{-\sigma} u \cedge{-\sigma} x_k$.
Thus (C1) implies that $u \cedge{-\sigma} x_i$ for any $0 \leq i \leq k-1$; in particular $v_{\sigma} \noedge u \cedge{-\sigma} x_0 \in W$, therefore $u \in X_{-\sigma} \cup Y_{-\sigma}$.
Hence the claim for $N_{-\sigma}(x)$ holds, therefore the proof is concluded.
\end{proof}
\begin{lem}
\label{lem:proof_main_step2_5}
In the above setting, if $\sigma \in \{+,-\}$ and $y \in Y_{\sigma}$, then $N_{-\sigma}(y) \subset V' \cup N_{-\sigma}(v_{\sigma})$.
\end{lem}
\begin{proof}
Let $u \in N_{-\sigma}(y)$.
Then $u \neq v_{\sigma}$ and $u \neq v_{-\sigma}$ by Lemma \ref{lem:proof_main_step2_1} and (\ref{eq:proof_main_property_edges}).
Now we have $y \cedge{\sigma} x$ for some $x \in X_{\sigma} \cup W$, and $x \cedge{-\sigma} w$ for some $w \in X_{\sigma} \cup W \cup v_{-\sigma}$.
Thus $u \cedge{-\sigma} y \cedge{\sigma} x \cedge{-\sigma} w$, therefore (C2) implies that $u \cedge{-\sigma} x$ (since $y \nocedge{-\sigma} w$ by (\ref{eq:proof_main_property_edges})).
Now since $x \in X_{\sigma} \cup W$, the claim follows from Lemmas \ref{lem:proof_main_step2_3} and \ref{lem:proof_main_step2_4}.
\end{proof}

Owing to these results, here we prove that $S(G) \neq \emptyset$ in the case $V' \neq V$.
By the induction hypothesis, there is a vertex $v \in S(G|_{V'})$.
By (\ref{eq:proof_main_property_simplicial}), we have $v \in W \cup X_{\sigma}$ for some $\sigma \in \{+,-\}$.
From now, we show that $v \in S(G)$.\\
\quad
For the condition (S1), we show that $u \cedge{\tau} u'$ if $\tau \in \{+,-\}$, $u,u' \in N_{\tau}(v)$ and $u \neq u'$.
This holds if $u,u' \in V'$ since $v \in S(G|_{V'})$; thus suppose that $u \not\in V'$ or $u' \not\in V'$, say $u \not\in V'$.
Then, since $v \in W \cup X_{\sigma}$, we have $u \in N_{\tau}(v_{-\tau})$ by Lemmas \ref{lem:proof_main_step2_3} and \ref{lem:proof_main_step2_4}.
Moreover, we have $N_{\tau}(v) \subset N_{\tau}\left[u\right]$ by Lemmas \ref{lem:proof_main_step2_2}, \ref{lem:proof_main_step2_3} and \ref{lem:proof_main_step2_4}.
Thus we have $u \cedge{\tau} u'$, as desired.
Hence the condition (S1) is satisfied.\\
\quad
For the condition (S2), first we show that $v \cedge{-\sigma} u$ if $u \cedge{-\sigma} u' \cedge{\sigma} v$.
This holds if $u,u' \in V'$ since $v \in S(G|_{V'})$; thus suppose that $u \not\in V'$ or $u' \not\in V'$.
Now if $u' \in V'$ and $u \not\in V'$, then we have $u' \in V' \setminus (Y_{-\sigma} \cup v_{-\sigma})$ by Lemmas \ref{lem:proof_main_step2_3} and \ref{lem:proof_main_step2_4}, therefore $u \in N_{-\sigma}(v_{\sigma})$ by Lemmas \ref{lem:proof_main_step2_3}, \ref{lem:proof_main_step2_4} and \ref{lem:proof_main_step2_5}.
Thus Lemma \ref{lem:proof_main_step2_2} implies that $v \cedge{-\sigma} u$ in this case.
On the other hand, if $u' \not\in V'$, then $u' \in N_{\sigma}(v_{-\sigma})$ by Lemmas \ref{lem:proof_main_step2_3} and \ref{lem:proof_main_step2_4}.
Now since $v \in X_{\sigma} \cup W$, we have $v \cedge{-\sigma} x$ for some $x \in X_{\sigma} \cup W \cup v_{-\sigma}$, while $x \cedge{\sigma} u'$ by Lemma \ref{lem:proof_main_step2_2}.
Thus we have $v \cedge{-\sigma} x \cedge{\sigma} u' \cedge{-\sigma} u$, therefore (C2) implies that $v \cedge{-\sigma} u$ in this case.
Hence we have $v \cedge{-\sigma} u$, as desired.\\
\quad
Secondly, we show that $v \cedge{\sigma} u$ if $u \cedge{\sigma} u' \cedge{-\sigma} v$.
This holds if $u,u' \in V'$ since $v \in S(G|_{V'})$; thus suppose that $u \not\in V'$ or $u' \not\in V'$.
Now if $u' \in V'$ and $u \not\in V'$, then we have $u' \in V' \setminus (Y_{\sigma} \cup v_{\sigma})$ by Lemmas \ref{lem:proof_main_step2_3} and \ref{lem:proof_main_step2_4}, therefore $u \in N_{\sigma}(v_{-\sigma})$ by Lemmas \ref{lem:proof_main_step2_3}, \ref{lem:proof_main_step2_4} and \ref{lem:proof_main_step2_5}.
Thus Lemma \ref{lem:proof_main_step2_2} implies that $v \cedge{\sigma} u$ in this case.
On the other hand, if $u' \not\in V'$, then $u' \in N_{-\sigma}(v_{\sigma})$ by Lemmas \ref{lem:proof_main_step2_3} and \ref{lem:proof_main_step2_4}.
Now we have $u \cedge{\sigma} u' \cedge{-\sigma} v_{\sigma} \cedge{\sigma} v$ by Lemma \ref{lem:proof_main_step2_1}, therefore (C2) implies that $v \cedge{\sigma} u$ in this case.
Thus we have $v \cedge{\sigma} u$, as desired.
Hence the condition (S2) holds.\\
\quad
Thus we have $S(G) \neq \emptyset$ if $V' \neq V$, as desired.

Moreover, by applying the above result, we have the following corollaries:
\begin{cor}
\label{cor:proof_main_step2_1}
In the above setting, suppose that $u \cedge{\sigma} v \cedge{-\sigma} w \cedge{\sigma} x$ and $u \noedge w$ for some $\sigma \in \{+,-\}$ and distinct vertices $u$, $v$, $w$, and $x$.
Then $S(G) \neq \emptyset$.
\end{cor}
\begin{proof}
By (C2), we have $u \cedge{\sigma} x$.
Now we apply the above argument, where $u$ and $w$ play the roles of $v_{\sigma}$ and $v_{-\sigma}$, respectively.
Then we have $x \not\in V'$, since $x \not\in W \cup (V \setminus N(v_{\tau}))$ for each $\tau \in \{+,-\}$.
Hence we have $S(G) \neq \emptyset$ by the above result, as desired.
\end{proof}
\begin{cor}
\label{cor:proof_main_step2_2}
In the above setting, suppose that $u \cedge{+} v \cedge{-} w$, $u \noedge w$ and $N_+\left[u\right] \cup N_-\left[w\right] \neq V$.
Then $S(G) \neq \emptyset$.
\end{cor}
\begin{proof}
We apply the above argument, where $u$ and $w$ plays the roles of $v_+$ and $v_-$, respectively.
In this setting, Lemma \ref{lem:proof_main_step2_1} implies that $V' \subset N_+\left[u\right] \cup N_-\left[w\right]$, therefore $V' \neq V$ by the assumption.
Thus we have $S(G) \neq \emptyset$ by the above result.
\end{proof}

\subsection{Proof of Theorem \ref{thm:characterization}, Final Step}
\label{subsec:proof_main_theorem_3}

Now it suffices to consider the case that condition (\ref{eq:proof_main_condition_1}) does not hold and we have $V' = V$ under the notations used in Section \ref{subsec:proof_main_theorem_2}.
Moreover, owing to Corollaries \ref{cor:proof_main_step2_1} and \ref{cor:proof_main_step2_2}, we may assume without loss of generality that
\begin{equation}
\label{eq:proof_main_condition_2}
\mbox{ if } u \cedge{+} u' \cedge{-} u'' \mbox{ and } u \noedge u'', \mbox{ then } N_-(u) = N_+(u'') = \emptyset \enspace,
\end{equation}
\begin{equation}
\label{eq:proof_main_condition_3}
\mbox{ if } u \cedge{+} u' \cedge{-} u'' \mbox{ and } u \noedge u'', \mbox{ then } N_+\left[u\right] \cup N_-\left[u''\right] = V \enspace.
\end{equation}
By these assumptions, we have $N_{-\sigma}(v_{\sigma}) = \emptyset$ for each $\sigma \in \{+,-\}$.
Moreover, we have the following results:
\begin{lem}
\label{lem:proof_main_step3_1}
In the above setting, if $\sigma \in \{+,-\}$ and $y \in Y_{\sigma}$, then we have $N_{-\sigma}(y) = \emptyset$.
\end{lem}
\begin{proof}
Assume contrary that $N_{-\sigma}(y) \neq \emptyset$.
Owing to construction of $Y_{\sigma}$, take an induced path $x_1x_2 \cdots x_k$ in $G_{-\sigma}$, $k \geq 1$, with $k$ minimal, such that $x_1 \in W$, $x_i \in X_{\sigma}$ for $2 \leq i \leq k$ and $y \cedge{\sigma} x_k$.
Put $x_0 = v_{-\sigma}$.
Then we have $y \cedge{\sigma} x_k \cedge{-\sigma} x_{k-1}$ and $N_{-\sigma}(y) \neq \emptyset$, while $y \noedge x_{k-1}$ by (\ref{eq:proof_main_property_edges}) and minimality of the $k$.
This contradicts the condition (\ref{eq:proof_main_condition_2}).
Hence the proof is concluded.
\end{proof}
\begin{lem}
\label{lem:proof_main_step3_2}
In the above setting, we have $\left[X_+,X_-\right] \cap E = \emptyset$, and $\left[X_{\sigma},W\right] \cap E_{\sigma} = \emptyset$ for each $\sigma \in \{+,-\}$.
\end{lem}
\begin{proof}
It suffices to show that $\left[X_{\sigma},W \cup X_{-\sigma}\right] \cap E_{\sigma} = \emptyset$ for each $\sigma \in \{+,-\}$.
Assume contrary that $x \cedge{\sigma} x'$ for some $x \in X_{\sigma}$ and $x' \in W \cup X_{-\sigma}$.
Then we have $x \cedge{\sigma} x' \cedge{-\sigma} v_{-\sigma}$ by Lemma \ref{lem:proof_main_step2_1}, while $x \noedge v_{-\sigma}$ by (\ref{eq:proof_main_property_edges}).
Moreover, we have $N_{-\sigma}(x) \neq \emptyset$ by construction of $X_{\sigma}$.
This contradicts (\ref{eq:proof_main_condition_2}).
Hence the claim holds.
\end{proof}
\begin{lem}
\label{lem:proof_main_step3_3}
In the above setting, for each $\sigma \in \{+,-\}$, let $W_\sigma$ denote the union of vertex sets of the connected components of $G|_W$ that are joined with $X_{\sigma}$ by an edge in $E_{-\sigma}$.
Then $W_+ \cap W_- = \emptyset$.
\end{lem}
\begin{proof}
Assume contrary that $G|_{W_+}$ and $G|_{W_-}$ involve a common connected component.
Then there are a vertex $w_0 \in X_+$, an induced path $w_1w_2 \cdots w_{k-1}$ in $G|_W$ and a vertex $w_k \in X_-$ such that $w_0 \cedge{-} w_1$ and $w_{k-1} \cedge{+} w_k$.
Now there exists an index $1 \leq i \leq k-1$ such that $w_{i-1} \cedge{-} w_i \cedge{+} w_{i+1}$.
We have $v_+ \cedge{+} w_{i-1}$ and $v_- \cedge{-} w_{i+1}$ by Lemma \ref{lem:proof_main_step2_1}, therefore $v_+ \cedge{+} w_{i+1}$ and $v_- \cedge{-} w_{i-1}$ by (C2).
Thus we have $w_{i-1},w_{i+1} \in W$, therefore $2 \leq i \leq k-2$ and $w_{i-1} \noedge w_{i+1}$ since $w_1 \cdots w_{k-1}$ is an induced path in $G$.
This contradicts (\ref{eq:proof_main_condition_2}).
Hence the proof is concluded.
\end{proof}
\begin{lem}
\label{lem:proof_main_step3_4}
In the above setting, if $v \in V$, $N_+(v) \neq \emptyset$ and $N_-(v) \neq \emptyset$, then the condition (S1) for $v$ implies the condition (S2) for $v$.
\end{lem}
\begin{proof}
Suppose that (S1) holds, and let $\sigma \in \{+,-\}$ and $u \cedge{-\sigma} w \cedge{\sigma} v$.
Then we have $N_{-\sigma}(v) \neq \emptyset$ by the assumption, therefore $u \edge v$ by (\ref{eq:proof_main_condition_2}).
Moreover, if $u \cedge{\sigma} v$, then $u,w \in N_{\sigma}(v)$ and $u \nocedge{\sigma} w$, contradicting the condition (S1).
Hence we have $u \cedge{-\sigma} v$, therefore the condition (S2) holds.
\end{proof}

Now suppose that $X_+ \neq \emptyset$ and $X_- \neq \emptyset$, therefore $W_+ \neq \emptyset$ and $W_- \neq \emptyset$.
Put $V_1 = V \setminus (W_- \cup X_-)$ and $G_1 = G|_{V_1}$.
Then $G_1$ is connected by Lemma \ref{lem:proof_main_step2_1}, while $V_1 \neq V$ and $G_1$ has an edge in $E_+$ and an edge in $E_-$.
Thus by the induction hypothesis and Lemma \ref{lem:SE_simplicial_in_border}, we have $v \in S(G_1)$, $N_{(G_1)_+}(v) \neq \emptyset$ and $N_{(G_1)_-}(v) \neq \emptyset$ for some $v \in V_1$.
Now we have $v \in (W \setminus W_-) \cup X_+$ by Lemma \ref{lem:proof_main_step3_1}, therefore $N_G(v) \cap (W_- \cup X_-) = \emptyset$ by Lemma \ref{lem:proof_main_step3_2} and the construction of $W_+$ and $W_-$.
Thus we have $N_G(v) = N_{G_1}(v)$, therefore condition (S1) for $v$ and $G$ holds since $v \in S(G_1)$.
Hence Lemma \ref{lem:proof_main_step3_4} implies that $v \in S(G)$ in this case.

Finally, suppose that $X_+ = \emptyset$ or $X_- = \emptyset$, say $X_- = \emptyset$.
Now we have the following property:
\begin{lem}
\label{lem:proof_main_step3_5}
In the above setting, $Y_- \cup v_-$ is a clique in $G_-$.
\end{lem}
\begin{proof}
By Lemma \ref{lem:proof_main_step2_1}, it suffices to show that $Y_-$ is a clique in $G_-$.
Let $y \in Y_-$.
Then, since $X_- = \emptyset$, we have $y \cedge{-} w$ for some $w \in W$ by construction of $Y_-$.
Thus we have $v_+ \cedge{+} w \cedge{-} y$, while $v_+ \noedge y$ and $(Y_- \setminus y) \cap N_+(v_+) = \emptyset$ by (\ref{eq:proof_main_property_edges}).
Therefore condition (\ref{eq:proof_main_condition_3}) implies that $Y_- \setminus y \subset N_-(y)$.
Hence the claim holds.
\end{proof}
Put $V_2 = V \setminus v_+$ and $G_2 = G|_{V_2}$.
Then we have $S(G_2) \neq \emptyset$ by the induction hypothesis.
More strongly, there is a $v \in S(G_2)$ such that $v \in W \cup X_+ \cup Y_+$.
In fact, since $G_2$ is connected (every vertex in $G_2$ is joined with $v_-$ by a path in $G_2$), this holds by Lemmas \ref{lem:SE_simplicial_in_border} and \ref{lem:proof_main_step3_1} if $G_2$ has an edge in $E_+$, and by Lemmas \ref{lem:chordal_link_clique} and \ref{lem:proof_main_step3_5} if $G_2$ has no edge in $E_+$.
Now we have $N_{G_+}(v) \subset V \setminus (Y_- \cup v_-)$ by Lemma \ref{lem:proof_main_step3_1}, while $N_-(v_+) = \emptyset$ and $V \setminus (Y_- \cup v_-) \subset N_{G_+}\left[v_+\right]$ by Lemma \ref{lem:proof_main_step2_1}.
Thus the assumption and condition (D) in Lemma \ref{lem:simplicial_expand} are satisfied, where $\{v_+\}$ plays the role of $V''$, therefore we have $v \in S(G)$ by that lemma.

Hence the proof of the \lq\lq if'' part of Theorem \ref{thm:characterization} is concluded.

\section{Special Cases}
\label{sec:special_case}

In this section, we apply Theorem \ref{thm:characterization} to characterize the SE graphs in some subclasses.
Let $G = (V,E)$ be a signed graph throughout this section.
First, we consider the case of signed graphs with four vertices:
\begin{prop}
\label{prop:FourVertices}
If $|V| = 4$, then $G$ is signed-eliminable if and only if one of the following conditions is satisfied:
\begin{description}
\item[(FV1)] $G_+$ or $G_-$ has a vertex of degree three.
\item[(FV2)] Both $G_+$ and $G_-$ are chordal, $G$ is not a mountain, and $G$ has no alternating $4$-path.
\end{description}
\end{prop}
\begin{proof}
For the \lq\lq only if'' part, suppose that conditions (C1)--(C3) in Theorem \ref{thm:characterization} are satisfied and (FV2) does not hold.
Then by (C1) and (C3), the failure of (FV2) implies that $G$ has an alternating $4$-path, therefore (FV1) follows from (C2).
On the other hand, for the \lq\lq if'' part, suppose that (FV1) or (FV2) holds.
Now if (FV2) holds, then all of (C1), (C2) and (C3) follow, since any hill with four vertices involves an alternating $4$-path.
Moreover, suppose that (FV1) holds.
Then (C1) holds since neither an induced cycle with four vertices nor its complement in $G$ has a vertex of degree three, and (C3) holds by the shape of mountains and hills.
Moreover, if $G$ has an alternating $4$-path $u \cedge{\sigma} v \cedge{-\sigma} w \cedge{\sigma} x$ with $\sigma \in \{+,-\}$, then neither $v$ nor $w$ has degree three in $G_+$ or $G_-$, therefore either $u$ or $x$ has degree three in $G_{\sigma}$.
This implies that $w \cedge{\sigma} u \cedge{\sigma} x$ or $u \cedge{\sigma} x \cedge{\sigma} v$, therefore (C2) holds.
Hence the proof is concluded.
\end{proof}
Note that a list of the non-SE graphs with four vertices is given in \cite{ANN}.

Secondly, we consider the case that the underlying graph $G$ is chordal:
\begin{prop}
\label{prop:characterization_GisChordal}
Suppose that $G$ is chordal (as a non-signed graph).
Then $G$ is signed-eliminable if and only if conditions (C2) and (C3) in Theorem \ref{thm:characterization} are satisfied.
\end{prop}
\begin{proof}
It suffices to show that both $G_+$ and $G_-$ are chordal if $G$ is chordal and $G$ satisfies (C2) and (C3).
Let $\sigma \in \{+,-\}$ and let $v_0v_1 \cdots v_{k-1}v_0$ be a cycle in $G_{\sigma}$ with $k \geq 4$.
Then, since $G$ is chordal, we have $v_i \edge v_j$ for some indices $i$ and $j$ with $j \neq i \pm 1$ (where indices are considered in modulo $k$).
The claim holds if $v_i \cedge{\sigma} v_j$, thus suppose that $v_i \cedge{-\sigma} v_j$.
Then we have $v_{i-1} \cedge{\sigma} v_i \cedge{-\sigma} v_j \cedge{\sigma} v_{j-1}$, therefore $v_{i-1} \cedge{\sigma} v_{j-1}$ by (C2).
Hence the claim holds.
\end{proof}
Moreover, we consider the case that the underlying graph $G$ has no independent set of size three:
\begin{prop}
\label{prop:characterization_noThreeIndep}
Suppose that $G$ has no three distinct vertices $u$, $u'$ and $u''$ such that $u \noedge u' \noedge u'' \noedge u$ (i.e.\ the independence number $\alpha(G)$ of $G$ is less than three).
Then $G$ is signed-eliminable if and only if condition (C2) in Theorem \ref{thm:characterization} and the following two conditions are satisfied:
\begin{description}
\item[(I1)] Both $G_+$ and $G_-$ has no cycle of length four or five which is an induced cycle in $G$.
\item[(I2)] $G$ contains no hill with five or six vertices as an induced subgraph.
\end{description}
\end{prop}
\begin{proof}
The \lq\lq only if'' part follows from Theorem \ref{thm:characterization}.
To prove the \lq\lq if'' part, we show that conditions (C1) and (C3) hold if (C2), (I1) and (I2) are satisfied.
For condition (C1), let $v_0v_1 \cdots v_{k-1}v_0$ be a cycle in $G_{\sigma}$, with $\sigma \in \{+,-\}$ and $k \geq 4$.
Now we have $v_i \edge v_j$ for some indices $i \neq j$ with $j \neq i \pm 1$ (where indices are considered in modulo $k$).
In fact, this follows from (I1) if $k \leq 5$, while this follows if $k \geq 6$ since now $\{v_0,v_2,v_4\}$ does not form an independent set by the assumption.
The condition (C1) holds if $v_i \cedge{\sigma} v_j$, thus suppose that $v_i \cedge{-\sigma} v_j$.
Now we have $v_{i-1} \cedge{\sigma} v_i \cedge{-\sigma} v_j \cedge{-\sigma} v_{j-1}$, therefore $v_{i-1} \cedge{\sigma} v_{j-1}$ by (C2).
Thus the condition (C1) is satisfied.\\
\quad
For condition (C3), note that any mountain $(v_1,v_2,\dots,v_n;w)$ has an independent set $\{v_1,v_n,w\}$, and any hill $(v_1,v_2,\dots,v_n;w_1,w_2)$ with $n \geq 5$ has an independent set $\{v_1,v_3,v_5\}$.
Thus by the assumption, these graphs do not appear in $G$ as induced subgraphs.
Moreover, (C2) implies that any hill with four vertices does not appear in $G$ as an induced subgraph.
Hence the condition (C3) follows from (I2), therefore the proof is concluded.
\end{proof}
\begin{cor}
\label{cor:characterization_complete}
Suppose that $G$ is a complete graph (as a non-signed graph).
Then $G$ is signed-eliminable if and only if for each $\sigma \in \{+,-\}$, $G_{\sigma}$ contains, as an induced subgraph, neither a simple path with four vertices, nor a pair of two disjoint edges such that no vertex of one edge is joined by an edge in $G_{\sigma}$ with a vertex of another edge.
\end{cor}
\begin{proof}
Since conditions (I1) and (I2) in Proposition \ref{prop:characterization_noThreeIndep} are always satisfied by the assumption, it suffices to show that (C2) is now equivalent to the condition in the statement.
First, if (C2) holds then the condition in the statement is satisfied, since the two kinds of subgraphs in the statement do not satisfy the condition (C2).
On the other hand, suppose that the condition in the statement holds and $G$ has an alternating $4$-path $u \cedge{\sigma} v \cedge{-\sigma} w \cedge{\sigma} x$.
Put $V' = \{u,v,w,x\}$ and $G' = G|_{V'}$.
Now if $u \cedge{-\sigma} x$, then the edges $uv$ and $wx$ in $G_{\sigma}$ form a pair as in the statement when $G'_{\sigma}$ has no more edge; $G'_{\sigma}$ is a simple path with four vertices when $G'_{\sigma}$ has just one more edge; and the edges $ux$ and $vw$ in $G_{-\sigma}$ form a pair as in the statement when $G'_{\sigma}$ has two more edges.
Thus the condition in the statement implies that $u \cedge{\sigma} x$.
Moreover, since $G'_{\sigma}$ is not a simple path with four vertices by the condition, we have either $u \cedge{\sigma} w$ or $v \cedge{\sigma} x$.
Thus the condition (C2) holds.
Hence the proof is concluded.
\end{proof}

\end{document}